\documentclass[12pt]{article}
\usepackage[centertags]{amsmath}
\usepackage{amsfonts}
\usepackage{amssymb}
\usepackage{latexsym}
\usepackage{amsthm}
\usepackage{newlfont}
\usepackage{graphicx}
\usepackage{listings}
\usepackage{booktabs}
\usepackage{abstract}
\date{}

\newlength{\defbaselineskip}
\newcommand{\setlinespacing}[1]%
           {\setlength{\baselineskip}{#1 \defbaselineskip}}

\newcommand{\actaqed}{\hfill $\actabox$}
{\medskip\noindent \textit{Proof of #1. }}%
{\actaqed \medskip}

\def\D{{\mathcal D}}
\def\A{{\mathcal A}}
\def\H{{\mathcal H}}
\def \Tr{\mathcal T}

\def \T{\mathbb T}
\def \<{\langle}
\def\>{\rangle}

\def \e{\epsilon}
\def \de{\delta}

\def\la{\lambda}

\def \sp{\operatorname{span}}

\def\bx{\mathbf x}

\def\bk{\mathbf k}

\def\bW{\mathbf W}

\def\Th{{\Theta}}
\def\U{{\mathcal U}}
\def\bt{\beta}

\newtheorem{Theorem}{Theorem}[section]
\newtheorem{Lemma}{Lemma}[section]
\newtheorem{Definition}{Definition}[section]
\newtheorem{Proposition}{Proposition}[section]

\newtheorem{Corollary}{Corollary}[section]
\numberwithin{equation}{section}

\newcommand{\be}{\begin{equation}}
\newcommand{\ee}{\end{equation}}

\begin{document}

\title{{Sparse approximation by greedy algorithms} }
\author{V. Temlyakov \thanks{ University of South Carolina and Steklov Institute of Mathematics. Research was supported by NSF grant DMS-1160841 }}
\maketitle
\begin{abstract}
{It is a survey on recent results in constructive sparse approximation. Three directions are discussed here: (1) Lebesgue-type inequalities for greedy algorithms with respect to a special class of dictionaries, (2) constructive sparse approximation with respect to the trigonometric system, (3) sparse approximation with respect to dictionaries with tensor product structure. In all three cases constructive ways are provided for sparse approximation. The technique used is based on fundamental results from the theory of greedy approximation. In particular, results in the direction (1) are based on deep methods developed recently in compressed sensing. We present some of these results with detailed proofs. }
\end{abstract}

\section{Introduction}  

The paper is a survey on recent breakthrough results in constructive sparse approximation. In all cases discussed here the new technique is based on greedy approximation. The main motivation for the study of sparse approximation is that many real world signals can be well approximated by sparse ones.   Sparse approximation
automatically implies a need for nonlinear approximation, in particular, for greedy approximation. We give a brief description of a sparse approximation problem   and present a  discussion of the obtained results and their
relation to previous work. In Section 2 we concentrate on breakthrough results from \cite{LivTem} and \cite{Tsp}. In these papers we extended a fundamental result of Tong Zhang \cite{Zh} on the Lebesgue-type inequality for the RIP dictionaries in a Hilbert space (see Theorem \ref{T2.2} below) in several directions. We found new more general than the RIP conditions on a dictionary, which still guarantee the Lebesgue-type inequalities in a Hilbert space setting. We generalized these conditions to a Banach space setting and proved the Lebesgue-type inequalities for dictionaries satisfying those conditions. To illustrate the power of new conditions we applied this new technique to bases instead of redundant dictionaries. In particular, this technique gave very strong results for the trigonometric system. 

In a general setting we are working in a Banach space $X$ with a redundant system of elements $\D$ (dictionary $\D$). There is a solid justification of importance of a Banach space setting in numerical analysis in general and in sparse approximation in particular (see, for instance, \cite{Tbook}, Preface, and \cite{ST}). 
An element (function, signal) $f\in X$ is said to be $m$-sparse with respect to $\D$ if
it has a representation $f=\sum_{i=1}^mx_ig_i$,   $g_i\in \D$, $i=1,\dots,m$. The set of all $m$-sparse elements is denoted by $\Sigma_m(\D)$. For a given element $f_0$ we introduce the error of best $m$-term approximation
$
\sigma_m(f_0,\D) := \inf_{f\in\Sigma_m(\D)} \|f_0-f\|.
$
We are interested in the following fundamental problem of sparse approximation. 

{\bf Problem.} How to design a practical algorithm that builds sparse approximations comparable to best $m$-term approximations? 

In a general setting we study an algorithm (approximation method) $\A = \{A_m(\cdot,\D)\}_{m=1}^\infty$ with respect to a given dictionary $\D$. The sequence of mappings $A_m(\cdot,\D)$ defined on $X$ satisfies the condition: for any $f\in X$, $A_m(f,\D)\in \Sigma_m(\D)$. In other words, $A_m$ provides an $m$-term approximant with respect to $\D$. It is clear that for any $f\in X$ and any $m$ we have
$
\|f-A_m(f,\D)\| \ge \sigma_m(f,\D).
$
We are interested in such pairs $(\D,\A)$ for which the algorithm $\A$ provides approximation close to best $m$-term approximation. We introduce the corresponding definitions.
\begin{Definition}\label{D5.2} We say that $\D$ is an almost greedy dictionary with respect to $\A$ if there exist two constants $C_1$ and $C_2$ such that for any $f\in X$ we have
\begin{equation}\label{5.2}
\|f-A_{C_1m}(f,\D)\| \le C_2\sigma_m(f,\D).
\end{equation}
\end{Definition}
If $\D$ is an almost greedy dictionary with respect to $\A$ then $\A$ provides almost ideal sparse approximation. It provides $C_1m$-term approximant as good (up to a constant $C_2$) as ideal $m$-term approximant for every $f\in X$. In the case $C_1=1$ we call $\D$ a greedy dictionary.
We also need a more general definition. Let $\phi(u)$ be a  function such that 
$\phi(u)\ge 1$. 
\begin{Definition}\label{D5.3} We say that $\D$ is a $\phi$-greedy dictionary with respect to $\A$ if there exists a constant $C_3$ such that for any $f\in X$ we have
\begin{equation}\label{5.3}
\|f-A_{\phi(m)m}(f,\D)\| \le C_3\sigma_m(f,\D).
\end{equation}
\end{Definition}

If $\D=\Psi$ is a basis then in the above definitions we replace dictionary by basis. Inequalities of the form (\ref{5.2}) and (\ref{5.3}) are called the Lebesgue-type inequalities. 

In the above setting the quality criterion of the algorithm $\A$ is based on the Lebesgue-type inequalities, which hold for every individual $f\in X$. In classical approximation theory very often we use as a quality criterion of the algorithm $\A$ its performance  on a given class $F$. In this case we compare
$$
e_m(F,\A,\D):=\sup_{f\in F}\|f-A_m(f,\D)\|
$$
with 
$$
\sigma_m(F,\D):=\sup_{f\in F}\sigma_m(f,\D).
$$
We discuss this setting in Sections 5 and 6. 

In the case $\A=\{G_m(\cdot,\Psi)\}_{m=1}^\infty$ is the Thresholding Greedy Algorithm (TGA) the theory of greedy and almost greedy bases is well developed (see \cite{Tbook}). We remind that in case of a normalized basis $\Psi=\{\psi_k\}_{k=1}^\infty$ of a Banach space $X$ the TGA at the $m$th iteration 
gives an approximant  
$$
G_m(f,\Psi):=\sum_{j=1}^m c_{k_j}\psi_{k_j},\quad f=\sum_{k=1}^\infty c_k\psi_k, \quad |c_{k_1}|\ge |c_{k_2}| \ge \dots.
$$
In particular, it is known (see \cite{Tbook}, p. 17) that the univariate Haar basis is a greedy basis with respect to TGA for all $L_p$, $1<p<\infty$. Also, it is known that the TGA does not work well with respect to the trigonometric system.

We demonstrated in the paper \cite{Tsp} that the Weak Chebyshev Greedy Algorithm (WCGA) which we define momentarily is a solution to the above problem for a special class of dictionaries.  
 
Let $X$ be a real Banach space with norm $\|\cdot\|:=\|\cdot\|_X$. We say that a set of elements (functions) $\D$ from $X$ is a dictionary  if each $g\in \D$ has norm   one ($\|g\|=1$),
and the closure of $\sp \D$ is $X$.
For a nonzero element $g\in X$ we let $F_g$ denote a norming (peak) functional for $g$:
$
\|F_g\|_{X^*} =1,\qquad F_g(g) =\|g\|_X.
$
The existence of such a functional is guaranteed by the Hahn-Banach theorem.

Let
$t\in(0,1]$ be a given weakness parameter. We define the Weak Chebyshev Greedy Algorithm (WCGA) (see \cite{T15}) as a generalization for Banach spaces of the Weak Orthogonal Matching Pursuit (WOMP). In a Hilbert space the WCGA coincides with the WOMP. The WOPM is very popular in signal processing, in particular, in compressed sensing. In case $t=1$,   WOMP is called  Orthogonal Matching Pursuit (OMP).

 {\bf Weak Chebyshev Greedy Algorithm (WCGA).}
Let $f_0$ be given. Then for each $m\ge 1$ we have the following inductive definition.

(1) $\varphi_m :=\varphi^{c,t}_m \in \D$ is any element satisfying
$$
|F_{f_{m-1}}(\varphi_m)| \ge t\sup_{g\in\D}  | F_{f_{m-1}}(g)|.
$$

(2) Define
$
\Phi_m := \Phi^t_m := \sp \{\varphi_j\}_{j=1}^m,
$
and define $G_m := G_m^{c,t}$ to be the best approximant to $f_0$ from $\Phi_m$.

(3) Let
$
f_m := f^{c,t}_m := f_0-G_m.
$

The trigonometric system is a classical system that is known to be difficult to study. In \cite{Tsp} we study among other problems the problem of nonlinear sparse approximation with respect to it. Let  ${\mathcal R}{\mathcal T}$ denote the real trigonometric system 
$1,\sin 2\pi x,\cos 2\pi x, \dots$ on $[0,1]$ and let ${\mathcal R}{\mathcal T}_p$ to be its version normalized in $L_p([0,1])$. Denote ${\mathcal R}{\mathcal T}_p^d := {\mathcal R}{\mathcal T}_p\times\cdots\times {\mathcal R}{\mathcal T}_p$ the $d$-variate trigonometric system. We need to consider the real trigonometric system because the algorithm WCGA is well studied for the real Banach space. In order to illustrate performance of the WCGA we   discuss in this section  the above mentioned problem for the trigonometric system. 
   We proved in \cite{Tsp}  the following 
Lebesgue-type inequality for the WCGA.
\begin{Theorem}\label{T1.2} Let $\D$ be the normalized in $L_p$, $2\le p<\infty$, real $d$-variate trigonometric
system. Then    
for any $f_0\in L_p$ the WCGA with weakness parameter $t$ gives
\begin{equation}\label{I1.4}
\|f_{C(t,p,d)m\ln (m+1)}\|_p \le C\sigma_m(f_0,\D)_p .
\end{equation}
\end{Theorem}
The Open Problem 7.1 (p. 91) from \cite{T18} asks if (\ref{I1.4}) holds without an extra 
$\ln (m+1)$ factor. Theorem \ref{T1.2} is the first result on the Lebesgue-type inequalities for the WCGA with respect to the trigonometric system. It provides a progress in solving the above mentioned open problem, but the problem is still open. 

We note that properties of a given basis with respect to TGA and WCGA could be very different. For instance, the class of quasi-greedy bases (with respect to TGA), that is the class of bases $\Psi$ for which $G_m(f,\Psi)$ converges for each $f\in X$, is a rather narrow subset of all bases. It is close in a certain sense to the set of unconditional bases. The situation is absolutely different for the WCGA. If $X$ is uniformly smooth then WCGA converges for each $f\in X$ with respect to any dictionary in $X$ (see \cite{Tbook}, Ch.6).

Theorem \ref{T1.2} shows that the WCGA is very well designed for the trigonometric system. We show in \cite{Tsp}  that an analog of (\ref{I1.4}) holds for 
uniformly bounded orthogonal systems.  The proof of Theorem \ref{T1.2} uses technique developed in compressed sensing for proving the Lebesgue-type inequalities for redundant dictionaries with special properties. First results on Lebesgue-type inequalities were proved for incoherent dictionaries (see \cite{Tbook} for a detailed discussion). Then a number of results were proved for dictionaries satisfying the Restricted Isometry Property (RIP) assumption. 
 The incoherence assumption on a dictionary is stronger than the   RIP assumption. The corresponding Lebesgue-type inequalities for the Orthogonal Matching Pursuit (OMP) under RIP assumption were not known for a while. As a result new greedy-type algorithms were introduced and exact recovery of sparse signals and the Lebesgue-type inequalities were proved for these algorithms: the Regularized  Orthogonal Matching Pursuit (see \cite{NV}), Compressive Sampling Matching Pursuit (CoSaMP) (see \cite{NT}), and the Subspace Pursuit (SP) (see \cite{DM}). The OMP is simpler than CoSaMP and SP, however, at the time of invention of CoSaMP and SP these algorithms provided exact recovery of sparse signals and the Lebesgue-type inequalities for dictionaries satisfying the Restricted Isometry Property (RIP) (see \cite{NT} and \cite{DM}). The corresponding results for the OMP were not known at that time. Later,  a breakthrough result in this direction was obtained by Zhang \cite{Zh}. In particular, he proved that if $\D$ satisfies RIP then the OMP recovers exactly all $m$-sparse signals within  $Cm$ iterations.  
In \cite{LivTem} and \cite{Tsp}  we developed Zhang's technique to obtain recovery results and the Lebesgue-type inequalities in the Banach space setting.

The above Theorem \ref{T1.2} guarantees that the WCGA works very well for each individual function $f$. It is a constructive method, which provides after $\asymp m\ln m$ iterations an error comparable to $\sigma_m(f,\D)$. Here are two important points. First, in order to guarantee a rate of decay of errors $\|f_n\|$ of the WCGA we would like to know how smoothness assumptions on $f_0$ affect the rate of decay of  $\sigma_m(f_0,\D)$. Second, if, as we believe, one cannot get rid of $\ln m$ in Theorem \ref{T1.2} then it would be nice to find a constructive method, which provides on a certain smoothness class the order of best $m$-term approximation after $m$ iterations. Thus, as a complement to Theorem \ref{T1.2} we would like to obtain results, which relate rate of decay of $\sigma_m(f,\Tr^d)_p$ to some smoothness type properties of 
$f$. In Section 5 we concentrate on constructive methods of $m$-term approximation.  We measure smoothness in terms of mixed derivative and mixed difference. We note that the function classes with bounded mixed derivative are not only interesting and challenging object for approximation theory but they are important in numerical computations.

We discuss here the problem of sparse approximation. This problem is closely connected with the problem of recovery of sparse functions (signals). In the sparse recovery problem we assume that an unknown function $f$ is sparse with respect to a given dictionary and we want to recover it. This problem was a starting point for the compressed sensing theory (see \cite{Tbook}, Ch.5). In particular, the celebrated contribution of the work of Candes, Tao and Donoho was to show that the recovery can be done by the $\ell_1$ minimization algorithm. We stress that $\ell_1$ minimization algorithm works for the exact recovery of sparse signals. It does not provide sparse approximation. The greedy-type algorithms discussed in this paper provide sparse approximation, satisfying the Lebesgue-type inequalities. It is clear that the Lebesgue-type inequalities (\ref{5.2}) and (\ref{5.3}) guarantee exact recovery of sparse signals. 

\section{Lebesgue-type inequalities. General results.}
  
   A very important advantage of the WCGA  is its convergence and rate of convergence properties. The WCGA is well defined for all $m$. Moreover, it is known (see \cite{T15} and \cite{Tbook}) that the WCGA with weakness parameter $t\in(0,1]$ converges for all $f_0$ in all uniformly smooth Banach spaces with respect to any dictionary. That is, when $X$ is a real Banach space and the modulus of smoothness of $X$ is defined as follows
\begin{equation}\label{1.4}
\rho(u):=\frac{1}{2}\sup_{x,y;\|x\|= \|y\|=1}\left|\|x+uy\|+\|x-uy\|-2\right|,
\end{equation}
then the uniformly smooth Banach space is the one with $\rho(u)/u\to 0$ when $u\to 0$.

We discuss here the Lebesgue-type inequalities for the WCGA with weakness parameter $t\in(0,1]$. This discussion is based on papers \cite{LivTem} and \cite{Tsp}. For notational convenience we consider here a countable dictionary $\D=\{g_i\}_{i=1}^\infty$. The following assumptions {\bf A1} and {\bf A2} were used in \cite{LivTem}.
 For a given $f_0$ let sparse element (signal)
 $$
 f:=f^\e=\sum_{i\in T}x_ig_i,\quad g_i\in\D,
 $$
 be such that $\|f_0-f^\e\|\le \e$ and $|T|=K$. For $A\subset T$ denote
 $$
 f_A:=f_A^\e := \sum_{i\in A}x_ig_i.
 $$
 
  {\bf A1.} We say that $f=\sum_{i\in T}x_ig_i$ satisfies the Nikol'skii-type $\ell_1X$ inequality with parameter $r$ if
 \begin{equation}\label{C1}
 \sum_{i\in A} |x_i| \le C_1|A|^{r}\|f_A\|,\quad A\subset T.
 \end{equation}
 We say that a dictionary $\D$ has the Nikol'skii-type $\ell_1X$ property with parameters $K$, $r$   if any $K$-sparse element satisfies the Nikol'skii-type
 $\ell_1X$ inequality with parameter $r$.

{\bf A2.}  We say that $f=\sum_{i\in T}x_ig_i$ has incoherence property with parameters $D$ and $U$ if for any $A\subset T$ and any $\Lambda$ such that $A\cap \Lambda =\emptyset$, $|A|+|\Lambda| \le D$ we have for any $\{c_i\}$
\begin{equation}\label{C2}
\|f_A-\sum_{i\in\Lambda}c_ig_i\|\ge U^{-1}\|f_A\|.
\end{equation}
We say that a dictionary $\D$ is $(K,D)$-unconditional with a constant $U$ if for any $f=\sum_{i\in T}x_ig_i$ with
$|T|\le K$ inequality (\ref{C2}) holds.

The term {\it unconditional} in {\bf A2} is justified by the following remark. The above definition of $(K,D)$-unconditional dictionary is equivalent to the following definition. Let $\D$ be such that any subsystem of $D$ distinct elements $e_1,\dots,e_D$ from $\D$ is linearly independent and for any $A\subset [1,D]$ with $|A|\le K$ and any coefficients $\{c_i\}$ we have
$$
\|\sum_{i\in A}c_ie_i\| \le U\|\sum_{i=1}^Dc_ie_i\|.
$$

It is convenient for us to use the following assumption {\bf A3} introduced in \cite{Tsp} which is a corollary of assumptions {\bf A1} and {\bf A2}. 

{\bf A3.} We say that $f=\sum_{i\in T}x_ig_i$ has $\ell_1$ incoherence property with parameters $D$, $V$, and $r$ if for any $A\subset T$ and any $\Lambda$ such that $A\cap \Lambda =\emptyset$, $|A|+|\Lambda| \le D$ we have for any $\{c_i\}$
\begin{equation}\label{C3}
\sum_{i\in A}|x_i| \le V|A|^r\|f_A-\sum_{i\in\Lambda}c_ig_i\|.
\end{equation}
A dictionary $\D$ has $\ell_1$ incoherence property with parameters $K$, $D$, $V$, and $r$ if for any $A\subset B$, $|A|\le K$, $|B|\le D$ we have for any $\{c_i\}_{i\in B}$
$$
\sum_{i\in A} |c_i| \le V|A|^r\|\sum_{i\in B} c_ig_i\|.
$$

It is clear that {\bf A1} and {\bf A2} imply {\bf A3} with $V=C_1U$. Also, {\bf A3} implies {\bf A1} with $C_1=V$ and {\bf A2} with $U=VK^r$. Obviously, we can restrict ourselves to $r\le 1$. 

We give a simple remark that widens the collection of dictionaries satisfying the above properties {\bf A1}, {\bf A2}, and {\bf A3}. 
\begin{Definition}\label{D<} Let $\D^1=\{g^1_i\}$ and $\D^2=\{g^2_i\}$ be countable dictionaries. We say that $\D^2$ $D$-dominates $\D^1$ (with a constant $B$) if for any set $\Lambda$, $|\Lambda|\le D$, of indices and any coefficients $\{c_i\}$ we have
$$
\|\sum_{i\in \Lambda} c_ig^1_i\| \le B \|\sum_{i\in \Lambda} c_ig^2_i\|.
$$
In such a case we write $\D^1 \prec \D^2$ or more specifically $\D^1 \le B\D^2$.

In the case $\D^1 \le E_1^{-1}\D^2$ and $\D^2 \le E_2\D^1$ we say that $\D^1$ and $\D^2$ are $D$-equivalent (with constants $E_1$ and $E_2$) and write $\D^1 \approx \D^2$ or more specifically $E_1\D^1 \le \D^2 \le E_2\D^1$. 
\end{Definition}

\begin{Proposition}\label{P2.1} Assume $\D^1$ has one of the properties {\bf A1} or {\bf A3}. 
If $\D^2$ $D$-dominates $\D^1$ (with a constant $B$) then $\D^2$ has the same property as $\D^1$: {\bf A1} with $C^2_1= C^1_1B$ or {\bf A3} with $V^2 = V^1B$. 
\end{Proposition}
\begin{proof} In both cases {\bf A1} and {\bf A3} the proof is the same. We demonstrate the case {\bf A3}. Let $f=\sum_{i\in T} x_ig_i^2$. Then by the {\bf A3} property of $\D^1$ we have
$$
\sum_{i\in A}|x_i| \le V^1|A|^r\|\sum_{i\in A} x_ig_i^1 - \sum_{i\in \Lambda} c_ig_i^1\| \le V^1B|A|^r\|\sum_{i\in A} x_ig_i^2 - \sum_{i\in \Lambda} c_ig_i^2\|.
$$
\end{proof}

\begin{Proposition}\label{P2.2} Assume $\D^1$ has  the property {\bf A2}. 
If $\D^1$ and $\D^2$ are $D$-equivalent (with constants $E_1$ and $E_2$) then $\D^2$ has  property   {\bf A2} with $U^2= U^1E_2/E_1$.   
\end{Proposition}
\begin{proof} Let $f=\sum_{i\in T} x_ig_i^2$. Then by $\D^1 \approx \D^2$ and  the {\bf A2} property of $\D^1$ we have
$$
\|\sum_{i\in A} x_ig_i^2 - \sum_{i\in \Lambda} c_ig_i^2\| \ge E_1\|\sum_{i\in A} x_ig_i^1 - \sum_{i\in \Lambda} c_ig_i^1\|
$$
$$
\ge  E_1 (U^1)^{-1}\|\sum_{i\in A} x_ig_i^1\| \ge  (E_1/E_2) (U^1)^{-1}\|\sum_{i\in A} x_ig_i^2\|.
$$
\end{proof}

We now proceed to main results of \cite{LivTem} and \cite{Tsp} on the WCGA with respect to redundant dictionaries. The following Theorem \ref{T2.1} from \cite{Tsp} in the case $q=2$ was proved in \cite{LivTem}. 

 \begin{Theorem}\label{T2.1} Let $X$ be a Banach space with $\rho(u)\le \gamma u^q$, $1<q\le 2$. Suppose $K$-sparse $f^\e$ satisfies {\bf A1}, {\bf A2} and $\|f_0-f^\e\|\le \e$. Assume that $rq'\ge 1$. Then the WCGA with weakness parameter $t$ applied to $f_0$ provides
$$
\|f_{C(t,\gamma,C_1)U^{q'}\ln (U+1) K^{rq'}}\| \le C\e\quad\text{for}\quad K+C(t,\gamma,C_1)U^{q'}\ln (U+1) K^{rq'}\le D
$$
with an absolute constant $C$.
\end{Theorem}

 It was pointed out in \cite{LivTem} that Theorem \ref{T2.1} provides a corollary for Hilbert spaces that gives sufficient conditions somewhat weaker than the known RIP conditions on $\D$ for the Lebesgue-type inequality to hold. We formulate the corresponding definitions and results. 
  Let $\D$ be the Riesz dictionary with depth $D$ and parameter $\delta\in (0,1)$. This class of dictionaries is a generalization of the class of classical Riesz bases. We give a definition in a general Hilbert space (see \cite{Tbook}, p. 306).
\begin{Definition}\label{D3.1} A dictionary $\D$ is called the Riesz dictionary with depth $D$ and parameter $\delta \in (0,1)$ if, for any $D$ distinct elements $e_1,\dots,e_D$ of the dictionary and any coefficients $a=(a_1,\dots,a_D)$, we have
\begin{equation}\label{3.3}
(1-\de)\|a\|_2^2 \le \|\sum_{i=1}^D a_ie_i\|^2\le(1+\de)\|a\|_2^2.
\end{equation}
We denote the class of Riesz dictionaries with depth $D$ and parameter $\delta \in (0,1)$ by $R(D,\de)$.
\end{Definition}
The term Riesz dictionary with depth $D$ and parameter $\delta \in (0,1)$ is another name for a dictionary satisfying the Restricted Isometry Property (RIP) with parameters $D$ and $\de$. The following simple lemma holds.
\begin{Lemma}\label{L3.1} Let $\D\in R(D,\de)$ and let $e_j\in \D$, $j=1,\dots, s$. For $f=\sum_{i=1}^s a_ie_i$ and $A \subset \{1,\dots,s\}$ denote
$$
S_A(f) := \sum_{i\in A} a_ie_i.
$$
If $s\le D$ then
$$
\|S_A(f)\|^2 \le (1+\de)(1-\de)^{-1} \|f\|^2.
$$
\end{Lemma}

Lemma \ref{L3.1} implies that if $\D\in R(D,\de)$ then it is $(D,D)$-unconditional with a constant $U=(1+\de)^{1/2}(1-\de)^{-1/2}$.

 \begin{Theorem}\label{T2.2} Let $X$ be a Hilbert space. Suppose $K$-sparse $f^\e$ satisfies  {\bf A2} and $\|f_0-f^\e\|\le \e$. Then the WOMP with weakness parameter $t$ applied to $f_0$ provides
$$
\|f_{C(t,U) K}\| \le C\e\quad\text{for}\quad K+C(t,U) K\le D
$$
with an absolute constant $C$.
\end{Theorem}
Theorem \ref{T2.2} implies the following corollaries.
 \begin{Corollary}\label{C2.1} Let $X$ be a Hilbert space. Suppose any $K$-sparse $f$ satisfies   {\bf A2}.   Then the WOMP with weakness parameter $t$ applied to $f_0$ provides
$$
\|f_{C(t,U) K}\| \le C\sigma_K(f_0,\D)\quad\text{for}\quad K+C(t,U) K\le D
$$
with an absolute constant $C$.
\end{Corollary}

 \begin{Corollary}\label{C2.2} Let $X$ be a Hilbert space. Suppose $\D\in R(D,\de)$.     Then the WOMP with weakness parameter $t$ applied to $f_0$ provides
$$
\|f_{C(t,\de) K}\| \le C\sigma_K(f_0,\D)\quad\text{for}\quad K+C(t,\de) K\le D
$$
with an absolute constant $C$.
\end{Corollary}

  We   emphasized in \cite{LivTem} that in Theorem \ref{T2.1} we impose our conditions on an individual function $f^\e$. It may happen that the dictionary does not have the Nikol'skii  $\ell_1X$ property and $(K,D)$-unconditionality but the given $f_0$ can be approximated by $f^\e$ which does satisfy assumptions {\bf A1} and {\bf A2}. Even in the case of a Hilbert space the above results from \cite{LivTem} add something new to the study based on the RIP property of a dictionary. First of all, Theorem \ref{T2.2} shows that it is sufficient to impose assumption {\bf A2} on $f^\e$ in order to obtain exact recovery and the Lebesgue-type inequality results. Second, Corollary \ref{C2.1} shows that the condition {\bf A2}, which is weaker than the RIP condition, is sufficient for exact recovery and the Lebesgue-type inequality results. Third, Corollary \ref{C2.2} shows that even if we impose our assumptions in terms of RIP we do not need to assume that $\de < \de_0$. In fact, the result works for all $\de<1$ with parameters depending on $\de$.
  
Theorem \ref{T2.1} follows from the combination of Theorems \ref{T2.3} and \ref{T2.4}.
In case $q=2$ these theorems were proved in \cite{LivTem} and in general case $q\in(1,2]$ -- in \cite{Tsp}. 

\begin{Theorem}\label{T2.3} Let $X$ be a Banach space with $\rho(u)\le \gamma u^q$, $1<q\le 2$. Suppose for a given $f_0$ we have $\|f_0-f^\e\|\le \e$ with $K$-sparse $f:=f^\e$ satisfying {\bf A3}. Then for any $k\ge 0$ we have for $K+m \le D$
$$
\|f_m\| \le \|f_k\|\exp\left(-\frac{c_1(m-k)}{K^{rq'}}\right) +2\e, \quad q':=\frac{q}{q-1},
$$
where $c_1:= \frac{t^{q'}}{2(16\gamma)^{\frac{1}{q-1}} V^{q'}}$.
\end{Theorem}
In all theorems that follow we assume $rq'\ge 1$.
\begin{Theorem}\label{T2.4} Let $X$ be a Banach space with $\rho(u)\le \gamma u^q$, $1<q\le 2$. Suppose $K$-sparse $f^\e$ satisfies {\bf A1}, {\bf A2} and $\|f_0-f^\e\|\le \e$. Then the WCGA with weakness parameter $t$ applied to $f_0$ provides
$$
\|f_{C'U^{q'}\ln (U+1) K^{rq'}}\| \le CU\e\quad\text{for}\quad K+C'U^{q'}\ln (U+1) K^{rq'}\le D
$$
with an absolute constant $C$ and $C' = C_2(q)\gamma^{\frac{1}{q-1}} C_1^{q'} t^{-q'}$.
\end{Theorem}
We formulate an immediate corollary of Theorem \ref{T2.4} with $\e=0$.
\begin{Corollary}\label{C2.3} Let $X$ be a Banach space with $\rho(u)\le \gamma u^q$. Suppose $K$-sparse $f$ satisfies {\bf A1}, {\bf A2}. Then the WCGA with weakness parameter $t$ applied to $f$ recovers it exactly after $C'U^{q'}\ln (U+1) K^{rq'}$ iterations under condition $K+C'U^{q'}\ln (U+1) K^{rq'}\le D$.
\end{Corollary}

We formulate versions of Theorem \ref{T2.4} with assumptions {\bf A1}, {\bf A2} replaced by a single assumption {\bf A3} and replaced by two assumptions {\bf A2} and {\bf A3}. The corresponding modifications in the proofs go as in the proof of Theorem \ref{T2.3}. 

\begin{Theorem}\label{T2.5} Let $X$ be a Banach space with $\rho(u)\le \gamma u^q$, $1<q\le 2$. Suppose $K$-sparse $f^\e$ satisfies {\bf A3}  and $\|f_0-f^\e\|\le \e$. Then the WCGA with weakness parameter $t$ applied to $f_0$ provides
$$
\|f_{C(t,\gamma,q)V^{q'}\ln (VK) K^{rq'}}\| \le CVK^r\e\quad\text{for}\quad K+C(t,\gamma,q)V^{q'}\ln (VK) K^{rq'}\le D
$$
with an absolute constant $C$ and $C(t,\gamma,q) = C_2(q)\gamma^{\frac{1}{q-1}}  t^{-q'}$.
\end{Theorem}

\begin{Theorem}\label{T2.6} Let $X$ be a Banach space with $\rho(u)\le \gamma u^q$, $1<q\le 2$. Suppose $K$-sparse $f^\e$ satisfies {\bf A2}, {\bf A3}  and $\|f_0-f^\e\|\le \e$. Then the WCGA with weakness parameter $t$ applied to $f_0$ provides
$$
\|f_{C(t,\gamma,q)V^{q'}\ln (U+1) K^{rq'}}\| \le CU\e\quad\text{for}\quad K+C(t,\gamma,q)V^{q'}\ln (U+1) K^{rq'}\le D
$$
with an absolute constant $C$ and $C(t,\gamma,q) = C_2(q)\gamma^{\frac{1}{q-1}}  t^{-q'}$.
\end{Theorem}

Theorems \ref{T2.5} and \ref{T2.3} imply the following analog of Theorem \ref{T2.1}.
  
  \begin{Theorem}\label{T2.7} Let $X$ be a Banach space with $\rho(u)\le \gamma u^q$, $1<q\le 2$. Suppose $K$-sparse $f^\e$ satisfies {\bf A3}  and $\|f_0-f^\e\|\le \e$. Then the WCGA with weakness parameter $t$ applied to $f_0$ provides
$$
\|f_{C(t,\gamma,q)V^{q'}\ln (VK) K^{rq'}}\| \le C\e\quad\text{for}\quad K+C(t,\gamma,q)V^{q'}\ln (VK) K^{rq'}\le D
$$
with an absolute constant $C$ and $C(t,\gamma,q) = C_2(q)\gamma^{\frac{1}{q-1}}  t^{-q'}$.
\end{Theorem}

The following edition of Theorems \ref{T2.1} and \ref{T2.7} is also useful in applications. It follows from Theorems \ref{T2.6} and \ref{T2.3}.

 \begin{Theorem}\label{T2.8} Let $X$ be a Banach space with $\rho(u)\le \gamma u^q$, $1<q\le 2$. Suppose $K$-sparse $f^\e$ satisfies {\bf A2}, {\bf A3}  and $\|f_0-f^\e\|\le \e$. Then the WCGA with weakness parameter $t$ applied to $f_0$ provides
$$
\|f_{C(t,\gamma,q)V^{q'}\ln (U+1) K^{rq'}}\| \le C\e\quad\text{for}\quad K+C(t,\gamma,q)V^{q'}\ln (U+1) K^{rq'}\le D
$$
with an absolute constant $C$ and $C(t,\gamma,q) = C_2(q)\gamma^{\frac{1}{q-1}}  t^{-q'}$.
\end{Theorem}
 
 \section{Proofs}

 {\bf Proof of Theorem \ref{T2.3}.} We begin with a proof of Theorem \ref{T2.3}.
 \begin{proof}  Let
$$
 f:=f^\e=\sum_{i\in T}x_ig_i,\quad |T|=K,\quad g_i\in \D.
 $$
 Denote by $T^m$ the set of indices of $g_j\in \D$, $j\in T$, picked by the WCGA after $m$ iterations, $\Gamma^m := T\setminus T^m$.
 Denote by $A_1(\D)$ the closure in $X$ of the convex hull of the symmetrized dictionary $\D^\pm:=\{\pm g,g\in D\}$.   We will bound  $\|f_m\|$ from above.  Assume $\|f_{m-1}\|\ge \e$. Let $m>k$. We bound from below
 $$
 S_m:=\sup_{\phi\in A_1(\D)} |F_{f_{m-1}}(\phi)|.
 $$
 Denote $A_m:=\Gamma^{m-1}$. Then
 $$
 S_m \ge F_{f_{m-1}}(f_{A_m}/\|f_{A_m}\|_1),
 $$
 where $\|f_A\|_1 := \sum_{i\in A} |x_i|$. Next, by Lemma 6.9, p. 342, from \cite{Tbook} we obtain
 $$
 F_{f_{m-1}}(f_{A_m}) = F_{f_{m-1}}(f^\e) \ge \|f_{m-1}\|-\e.
 $$
 Thus
 \begin{equation}\label{3.4}
 S_m\ge \|f_{A_m}\|^{-1}_1(\|f_{m-1}\|-\e).
 \end{equation}
 From the definition of the modulus of smoothness we have for any $\la$
\begin{equation}\label{3.4'}
\|f_{m-1}-\la \varphi_m \| +\|f_{m-1}+\la \varphi_m \| \le 2\|f_{m-1}\|\left(1+\rho\left(\frac{\la}{\|f_{m-1}\|}\right)\right)
\end{equation}
and by (1) from the definition of the WCGA and Lemma 6.10 from \cite{Tbook}, p. 343, we get
$$
|F_{f_{m-1}}(\varphi_m)| \ge t \sup_{g\in \D} |F_{f_{m-1}}(g)| =
$$
$$
t\sup_{\phi\in A_1(\D)} |F_{f_{m-1}}(\phi)|=tS_m.
$$
Then either $F_{f_{m-1}}(\varphi_m)\ge tS_m$ or $F_{f_{m-1}}(-\varphi_m)\ge tS_m$. Both cases are treated in the same way. We demonstrate the case $F_{f_{m-1}}(\varphi_m)\ge tS_m$. We have for $\la \ge 0$
$$
\|f_{m-1}+\la \varphi_m\| \ge F_{f_{m-1}}(f_{m-1}+\la\varphi_m)\ge  \|f_{m-1}\| + \la tS_m.
$$
From here and from (\ref{3.4'}) we obtain
 $$
 \|f_m\| \le \|f_{m-1}-\la \varphi_m\| \le \|f_{m-1}\| +\inf_{\la\ge 0} (-\la t S_m + 2\|f_{m-1}\|\rho(\la/\|f_{m-1}\|)).
 $$ We discuss here the case $\rho(u) \le \gamma u^q$. Using (\ref{3.4}) we get
 $$
 \|f_m\| \le \|f_{m-1}\|\left(1-\frac{\la t}{\|f_{A_m}\|_1} +2\gamma \frac{\la^q}{\|f_{m-1}\|^q}\right) + \frac{\e \la t}{\|f_{A_m}\|_1}.
 $$
 Let $\la_1$ be a solution of
 $$
 \frac{\la t}{2\|f_{A_m}\|_1} = 2\gamma \frac{\la^q}{\|f_{m-1}\|^q},\quad \la_1 = \left(\frac{t\|f_{m-1}\|^q}{4\gamma  \|f_{A_m}\|_1}\right)^{\frac{1}{q-1}}.
 $$
 Our assumption (\ref{C3}) gives
 $$
 \|f_{A_m}\|_1= \|(f^\e-G_{m-1})_{A_m}\|_1 \le VK^r\|f^\e-G_{m-1}\|
 $$
 \begin{equation}\label{3.4a}
 \le VK^r(\|f_0-G_{m-1}\|+\|f_0-f^\e\|) \le VK^r(\|f_{m-1}\|+\e).
 \end{equation}
 Specify
 $$
 \la = \left(\frac{t \|f_{A_m}\|_1^{q-1}}{16\gamma (VK^{r})^q}\right)^{\frac{1}{q-1}}.
 $$
 Then, using $\|f_{m-1}\| \ge \e$ we get
 $$
 \left(\frac{\la}{\la_1}\right)^{q-1} = \frac{\|f_{A_m}\|^q_1}{4\|f_{m-1}\|^q (VK^{r})^q} \le 1
 $$
  and  obtain
  \begin{equation}\label{3.4b}
 \|f_m\| \le \|f_{m-1}\| \left( 1- \frac{t^{q'}}{2(16\gamma)^{\frac{1}{q-1}} (VK^r)^{q'}}\right) + \frac{\e t^{q'}}{(16\gamma)^{\frac{1}{q-1}} (VK^r)^{q'}}.
 \end{equation}
 Denote $c_1:= \frac{t^{q'}}{2(16\gamma)^{\frac{1}{q-1}} V^{q'}}$. Then
 $$
 \|f_m\| \le \|f_{k}\|\exp\left(-\frac{c_1(m-k)}{K^{rq'}}\right) + 2\e.
 $$
\end{proof}

{\bf Proof of Theorem \ref{T2.4}.} We proceed to a proof of Theorem \ref{T2.4}. Modifications of this proof which are in a style of the above proof of Theorem \ref{T2.3} give Theorems \ref{T2.5} and \ref{T2.6}.
  \begin{proof} We begin with a brief description of the structure of the proof. We are given $f_0$ and $f:=f^\e$ such that $\|f_0-f\|\le \e$ and $f$ is $K$-sparse satisfying {\bf A1} and {\bf A2}. We apply the WCGA to $f_0$ and control how many dictionary elements $g_i$ from the representation of $f$
  $$
  f:=f^\e:=\sum_{i\in T} x_ig_i
  $$
  are picked up by the WCGA after $m$ iterations. As above denote by 
  $T^m $ the set of indices $i\in T$ such that $g_i$ has been taken by the WCGA at one of the first $m$ iterations.  
  Denote  $\Gamma^m := T\setminus T^m$. It is clear that if $\Gamma^m=\emptyset$ then 
  $\|f_m\|\le \e$ because in this case $f\in \Phi_m$. 
  
  Our analysis goes as follows. For a residual $f_k$ we assume that $\Gamma^k$ is nonempty. Then we prove that after $N(k)$ iterations we arrive at a residual $f_{k'}$, $k'=k+N(k)$, such that either 
  \begin{equation}\label{*} 
  \|f_{k'}\| \le CU\e
  \end{equation}
  or
   \begin{equation}\label{*1} 
   |\Gamma^{k'}| < |\Gamma^k|- 2^{L-2}
\end{equation}
with some natural number $L$. An important fact is that for the number $N(k)$ of iterations we have a bound
 \begin{equation}\label{*2}
 N(k)\le \beta 2^{aL}, \qquad a:=rq'.
 \end{equation}
 
 Next, we prove that if we begin with $k=0$ and apply the above argument to the sequence of residuals $f_0$, $f_{k_1}$, ..., $f_{k_s}$, then after not more than $N:=2^{2a+1}\beta K^a$ iterations, we obtain either $\|f_N\|\le CU\e$ or $\Gamma^N=\emptyset$, which in turn implies that $\|f_N\|\le \e$. 
 
 We now proceed to the detailed argument. The following corollary of (\ref{C2}) will be often used: for $m\le D-K$ and $A\subset \Gamma^m$ we have
\begin{equation}\label{*3} 
 \|f_A\| \le U(\|f_m\|+\e).
 \end{equation}
 It follows from the fact that $f_A -f+G_m$ has the form $\sum_{i\in \Lambda}c_ig_i$ with $\Lambda$ satisfying $|A|+|\Lambda| \le D$, $A\cap\Lambda=\emptyset$, and from our assumption $\|f-f_0\|\le \e$. 
 
 The following lemma plays a key role in the proof. 
 \begin{Lemma}\label{L5.1} Let $f$ satisfy {\bf A1} and {\bf A2} and let $A\subset \Gamma^k$ be nonempty. Denote $B:=\Gamma^k\setminus A$. Then for any $m\in (k,D-K]$ we have either $\|f_{m-1}\|\le \epsilon$ or
  \begin{equation}\label{*4} 
  \|f_m\| \le \|f_{m-1}\|(1-u) + 2u(\|f_B\|+\e),
  \end{equation}
  where 
  $$
  u:=c_1|A|^{-rq'},\qquad c_1:=\frac{t^{q'}}{2(16\gamma)^{\frac{1}{q-1}}(C_1U)^{q'}},
  $$
  with $C_1$ and $U$ from {\bf A1} and {\bf A2}. 
  \end{Lemma}
  \begin{proof}   As above in the proof of Theorem \ref{T2.3} we bound $S_m$ from below. It is clear that $S_m\ge 0$.
 Denote $A(m) := A\cap \Gamma^{m-1}$. Then
 $$
 S_m \ge F_{f_{m-1}}(f_{A(m)}/\|f_{A(m)}\|_1).
 $$
   Next,
 $$
 F_{f_{m-1}}(f_{A(m)}) = F_{f_{m-1}}(f_{A(m)}+f_B-f_B).
 $$
 Then $f_{A(m)}+f_B=f^\e - f_\Lambda$ with $F_{f_{m-1}}(f_\Lambda)=0$. Moreover, it is easy to see that $F_{f_{m-1}}(f^\e)\ge\|f_{m-1}\|-\e$. Therefore,
 $$
 F_{f_{m-1}}(f_{A(m)}+f_B-f_B)\ge \|f_{m-1}\|-\e - \|f_B\|.
 $$
 Thus
 $$
 S_m\ge \|f_{A(m)}\|^{-1}_1\max(0,\|f_{m-1}\|-\e-\|f_B\|).
 $$
 By (\ref{C1}) we get
 $$
 \|f_{A(m)}\|_1 \le C_1 |A(m)|^{r}\|f_{A(m)}\| \le C_1 |A|^{r}\|f_{A(m)}\| .
 $$
 Then
 \begin{equation}\label{3.7}
 S_m \ge \frac{\|f_{m-1}\|-\|f_B\|-\e}{C_1 |A|^{r} \|f_{A(m)}\|}.
 \end{equation}
 From the definition of the modulus of smoothness we have for any $\la$
$$
\|f_{m-1}-\la \varphi_m \| +\|f_{m-1}+\la \varphi_m \| \le 2\|f_{m-1}\|\left(1+\rho\left(\frac{\la}{\|f_{m-1}\|}\right)\right)
$$
and by (1) from the definition of the WCGA and Lemma 6.10 from \cite{Tbook}, p. 343, we get
$$
|F_{f_{m-1}}(\varphi_m)| \ge t \sup_{g\in \D} |F_{f_{m-1}}(g)| =
$$
$$
t\sup_{\phi\in A_1(\D)} |F_{f_{m-1}}(\phi)|.
$$
 From here  we obtain
 $$
 \|f_m\| \le \|f_{m-1}\| +\inf_{\la\ge 0} (-\la t S_m + 2\|f_{m-1}\|\rho(\la/\|f_{m-1}\|)).
 $$
 We discuss here the case $\rho(u) \le \gamma u^q$. Using (\ref{3.7}) we get for any $\lambda\ge 0$
 $$
 \|f_m\| \le \|f_{m-1}\|\left(1-\frac{\la t}{C_1|A|^{r}\|f_{A(m)}\|} +2\gamma \frac{\la^q}{\|f_{m-1}\|^q}\right) + \frac{\la t(\|f_B\|+\e)}{C_1 |A|^{r} \|f_{A(m)}\|}.
 $$
 Let $\la_1$ be a solution of
 $$
 \frac{\la t}{2C_1|A|^{r}\|f_{A(m)}\|} = 2\gamma \frac{\la^q}{\|f_{m-1}\|^q},\quad \la_1 = \left(\frac{t\|f_{m-1}\|^q}{4\gamma C_1 |A|^{r}\|f_{A(m)}\|}\right)^{\frac{1}{q-1}}.
 $$
 Inequality (\ref{*3}) gives
 $$
 \|f_{A(m)}\| \le U(\|f_{m-1}\|+\e).
 $$
 Specify
 $$
 \la = \left(\frac{t \|f_{A(m)}\|^{q-1}}{16\gamma C_1 |A|^{r} U^q}\right)^{\frac{1}{q-1}}.
 $$
 Then $\la \le \la_1$ and we obtain
 \begin{equation}\label{3.8}
 \|f_m\| \le \|f_{m-1}\| \left( 1- \frac{t^{q'}}{2(16\gamma)^{\frac{1}{q-1}} (C_1 U |A|^{r})^{q'}}\right)+ \frac{ t^{q'} (\|f_B\|+\e)}{(16\gamma)^{\frac{1}{q-1}} (C_1 |A|^{r} U)^{q'}}.
 \end{equation}
\end{proof}
  
 For simplicity of notations we consider separately the case $|\Gamma^k|\ge 2$ and the case $|\Gamma^k|=1$. We begin with the generic case  $|\Gamma^k|\ge 2$. We apply Lemma \ref{L5.1} with different pairs $A_j,B_j$, which we now construct. Let $n$ be a natural number such that 
 $$
 2^{n-1} <|\Gamma^k| \le 2^n.
 $$
 For $j=1,2,\dots,n,n+1$ consider the following pairs of sets $A_j,B_j$: $A_{n+1}=\Gamma^k$, $B_{n+1}=\emptyset$; for $j\le n$, $A_j:=\Gamma^k\setminus B_j$ with $B_j\subset \Gamma^k$ is such that $|B_j|\ge |\Gamma^k|-2^{j-1}$ and for any set $J\subset \Gamma^k$ with $|J|\ge |\Gamma^k|-2^{j-1}$ we have
 $$
 \|f_{B_j}\| \le \|f_J\|.
 $$
 We note that the above definition implies that $|A_j|\le 2^{j-1}$ and that if for some $Q\subset \Gamma^k$ we have
 \begin{equation}\label{3.5}
 \|f_Q\| <\|f_{B_j}\|\quad \text{then}\quad |Q|< |\Gamma^k|-2^{j-1}.
 \end{equation}
Set $B_0:=\Gamma^k$. Note that property (\ref{3.5}) is obvious for $j=0$. 

Let $j_0\in[1,n]$ be an index such that if $j_0=1$ then $B_1 \neq \Gamma^k$ and if $j_0\ge 2$ then 
$$
B_1=B_2=\dots=B_{j_0-1} =\Gamma^k,  \qquad B_{j_0} \neq \Gamma^k.
$$
 For a given $b> 1$, to be specified later, denote by $L:=L(b)$ the index such that $(B_0:=\Gamma^k)$
 $$
 \|f_{B_0}\| < b\|f_{B_{j_0}}\|,
 $$
 $$
 \|f_{B_{j_0}}\| < b\|f_{B_{j_0+1}}\|,
 $$
 $$
 \dots
 $$
 $$
 \|f_{B_{L-2}}\| < b\|f_{B_{L-1}}\|,
 $$
 $$
 \|f_{B_{L-1}}\| \ge b\|f_{B_{L}}\|.
 $$
 Then
 \begin{equation}\label{3.6}
 \|f_{B_j}\| \le b^{L-1-j}\|f_{B_{L-1}}\|, \quad j=j_0,\dots,L,
 \end{equation}
 and
 \begin{equation}\label{3.6'}
 \|f_{B_0}\|=\dots=\|f_{B_{j_0-1}}\|\le b^{L-j_0}\|f_{B_{L-1}}\|.
 \end{equation}
Clearly, $L\le n+1$.
     
 Define $m_0:=\dots m_{j_0-1}:= k$ and, inductively,
 $$
 m_j=m_{j-1} + [\beta |A_j|^{rq'}],\quad j=j_0,\dots,L,
 $$
 where $[x]$ denotes the integer part of $x$. The parameter $\beta$ is any, which satisfies the following inequalities with $c_1$ from Lemma \ref{L5.1}
 \begin{equation}\label{3.6''}
 \beta \ge 1,\qquad e^{-c_1\beta/2} <1/2, \qquad 16Ue^{-c_1\beta/2} <1.
 \end{equation}
 We note that the inequality $\beta\ge 1$  implies that 
 $$
 [\beta |A_j|^{rq'}] \ge \beta |A_j|^{rq'}/2.
 $$
 Taking into account that $rq'\ge 1$ and $|A_j|\ge 1$ we obtain
 $$
 m_j\ge m_{j-1}+1.
 $$
 At iterations from $m_{j-1}+1$ to $m_j$ we apply Lemma \ref{L5.1} with $A=A_j$ and obtain from (\ref{*4}) that either $\|f_{m-1}\|\le \epsilon$ or
 $$
 \|f_m\| \le \|f_{m-1}\|(1-u) + 2u(\|f_{B_j}\|+\e),\quad u:= c_1|A_j|^{-rq'}.
 $$
 Using $1-u\le e^{-u}$ and $\sum_{k=0}^\infty (1-u)^k = 1/u$ we derive from here
 \begin{equation}\label{3.6a}
 \|f_{m_j}\| \le \|f_{m_{j-1}}\|e^{-c_1\beta/2}+ 2 (\|f_{B_j}\|+\e).
 \end{equation}
 We continue it up to $j=L$. Denote $\eta:=e^{-c_1\beta/2}$. Then either $\|f_{m_L}\|\le \epsilon$ or
 $$
 \|f_{m_L}\| \le \|f_k\|\eta^{L-j_0+1} +2 \sum_{j=j_0}^L (\|f_{B_j}\|+\e)\eta^{L-j}.
 $$
 We bound the $\|f_k\|$. It follows from the definition of $f_k$ that $\|f_k\|$ is the error of best approximation of $f_0$ by the subspace $\Phi_k$. Representing $f_0=f+f_0-f$ we see that $\|f_k\|$ is not greater than  the error of best approximation of $f$ by the subspace $\Phi_k$ plus $\|f_0-f\|$. This implies $\|f_k\|\le \|f_{B_0}\| +\e$. Therefore we continue
 $$
 \le (\|f_{B_0}\|+\e)\eta^{L-j_0+1} +2\sum_{j=j_0}^L (\|f_{B_{L-1}}\|(\eta b)^{L-j} b^{-1} + \e\eta^{L-j})
 $$
 $$
 \le b^{-1}\|f_{B_{L-1}}\|\left( (\eta b)^{L-j_0+1} + 2 \sum_{j=j_0}^L (\eta b)^{L-j}\right) + \frac{2\e}{1-\eta}.
 $$
Our choice of $\beta$ guarantees $\eta< 1/2$. Choose $b=\frac{1}{2\eta}$. Then
 \begin{equation}\label{3.10}
 \|f_{m_L}\| \le \|f_{B_{L-1}}\|8  e^{-c_1\beta/2}+4\e.
 \end{equation}

 By (\ref{*3}) we get
 $$
 \|f_{\Gamma^{m_L}}\| \le U(\|f_{m_L}\|+\e) \le U(\|f_{B_{L-1}}\|8 e^{-c_1\beta/2} + 5\e).
 $$
 If $\|f_{B_{L-1}}\|\le 10U \e$ then by (\ref{3.10})
 $$
 \|f_{m_L}\| \le CU\e,\qquad C=44.
 $$
 If $\|f_{B_{L-1}}\|\ge 10U\e$ then
 by our choice of $\beta$ we have $16U e^{-c_1\beta/2 }<1$ and
 $$
 U(\|f_{B_{L-1}}\|8e^{-c_1\beta/2} +5\e)< \|f_{B_{L-1}}\|.
 $$
 Therefore
 $$
 \|f_{\Gamma^{m_L}}\|<\|f_{B_{L-1}}\|.
 $$
 This implies
 $$
 |\Gamma^{m_L}| < |\Gamma^k| - 2^{L-2}.
 $$
 
 In the above proof our assumption $j_0\le n$ is equivalent to the assumption that $B_n\neq \Gamma^k$. We now consider the case $B_n=\Gamma^k$ and, therefore, $B_j=\Gamma^k$, $j=0,1,\dots,n$. This means that $\|f_{\Gamma^k}\|\le \|f_J\|$ for any $J$ with $|J|\ge |\Gamma^k|-2^{n-1}$. Therefore, if for some $Q\subset \Gamma^k$ we have 
\begin{equation}\label{3.5n}
 \|f_Q\| <\|f_{\Gamma^k}\|\quad \text{then}\quad |Q|< |\Gamma^k|-2^{n-1}.
 \end{equation}
 In this case we set $m_0:=k$ and 
 $$
 m_1:= k+[\beta |\Gamma^k|^{rq'}].
 $$
 Then by Lemma \ref{L5.1} with $A=\Gamma^k$   we obtain as in (\ref{3.6a})
 \begin{equation}\label{3.6b}
 \|f_{m_1}\| \le \|f_{m_0}\|e^{-c_1\beta/2}+ 2\e \le \|f_{\Gamma^k}\|e^{-c_1\beta/2} +3\e.
 \end{equation}
 By (\ref{*3}) we get
 $$
 \|f_{\Gamma^{m_1}}\| \le U(\|f_{m_1}\|+\e) \le U(\|f_{\Gamma^k}\| e^{-c_1\beta/2} + 4\e) .
 $$
 If $\|f_{\Gamma^k}\|\le 8U \e$ then  by (\ref{3.6b})
 $$
 \|f_{m_1}\| \le 7U\e.
 $$
  If $\|f_{\Gamma^k}\|\ge 8U\e$ then
 by our choice of $\beta$ we have $2U e^{-c_1\beta/2 }<1$ and
 \begin{equation}\label{3.6c}
 \|f_{\Gamma^{m_1}}\| \le U(\|f_{\Gamma^k}\|e^{-c_1\beta/2} +4\e)< \|f_{\Gamma^k}\|  .
 \end{equation}
 This implies
 $$
 |\Gamma^{m_1}| < |\Gamma^k| - 2^{n-1}.
 $$ 
 
 It remains to consider the case $|\Gamma^k|=1$. By the above argument, where we used Lemma \ref{L5.1} with $A=\Gamma^k$ we obtain (\ref{3.6c}). In the case $|\Gamma^k|=1$ inequality (\ref{3.6c}) implies $\Gamma^{m_1}=\emptyset$, which completes the proof in this case. 
 
   We now complete the proof of Theorem \ref{T2.4}. We begin with $f_0$ and apply the above argument (with $k=0$). As a result we either get the required inequality or we reduce the cardinality of support of $f$ from $|T|=K$ to $|\Gamma^{m_{L_1}}|<|T|-2^{L_1-2}$ (the WCGA picks up at least $2^{L_1-2}$ dictionary elements $g_i$ from the representation of $f$), $m_{L_1}\le \beta 2^{aL_1}$, $a:=rq'$. We continue the process and build a sequence $m_{L_j}$ such that $m_{L_j}\le \beta 2^{aL_j}$ and after $m_{L_j}$ iterations we reduce the support by at least $2^{L_j-2}$. We also note that $m_{L_j}\le \beta 2^{2a} K^{a}$. We continue this process till the following inequality is satisfied for the first time
 \begin{equation}\label{3.11}
  m_{L_1}+\dots+m_{L_s}  \ge 2^{2a}\beta K^{a}.
 \end{equation}
 Then, clearly,
 $$
  m_{L_1}+\dots+m_{L_s} \le 2^{2a+1}\beta  K^{a}.
 $$
 Using the inequality
 $$
 (a_1+\cdots +a_s)^{\theta} \le a_1^\theta+\cdots +a_s^\theta,\quad a_j\ge 0,\quad \theta\in (0,1]
 $$
 we derive from (\ref{3.11}) 
 $$
 2^{L_1-2}+\dots+2^{L_s-2} \ge \left(2^{a(L_1-2)}+\dots+2^{a(L_s-2)}\right)^{\frac{1}{a}}
 $$
 $$
 \ge
 2^{-2}\left(2^{aL_1}+\dots+2^{aL_s}\right)^{\frac{1}{a}}
 $$
 $$
 \ge 2^{-2}\left((\beta)^{-1}(m_{L_1}+\dots+m_{L_s})\right)^{\frac{1}{a}}\ge K.
 $$
 Thus, after not more than $N:=2^{2a+1}\beta  K^{a}$ iterations we either get the required inequality or we recover $f$ exactly (the WCGA picks up all the dictionary elements $g_i$ from the representation of $f$) and then $\|f_N\| \le \|f_0-f\|\le \e$.

 \end{proof}
 
 {\bf Proof of Theorem \ref{T2.5}.} We begin with a version of Lemma \ref{L5.1} that is used in this proof. 
 \begin{Lemma}\label{L5.2} Let $f$ satisfy {\bf A3} and let $A\subset \Gamma^k$ be nonempty. Denote $B:=\Gamma^k\setminus A$. Then for any $m\in (k,D-K]$ we have either $\|f_{m-1}\|\le \epsilon$ or
  \begin{equation}\label{A3.1} 
  \|f_m\| \le \|f_{m-1}\|(1-u) + 2u(\|f_B\|+\e),
  \end{equation}
  where 
  $$
  u:=c_2|A|^{-rq'},\qquad c_2:=\frac{t^{q'}}{2(16\gamma)^{\frac{1}{q-1}}V^{q'}},
  $$
  with $r$ and $V$  from {\bf A3}. 
  \end{Lemma}
  \begin{proof} The proof is a combination of proofs of Theorem \ref{T2.3} and Lemma \ref{L5.1}. As in the proof of Lemma \ref{L5.1} we denote $A(m):=A\cap\Gamma^{m-1}$ and get
  $$
 S_m\ge \|f_{A(m)}\|^{-1}_1\max(0,\|f_{m-1}\|-\e-\|f_B\|).
 $$
 From here in the same way as in the proof of Theorem \ref{T2.3} we obtain for any $\la\ge0$
 $$
 \|f_m\| \le \|f_{m-1}\|\left(1-\frac{\la t}{\|f_{A(m)}\|_1} +2\gamma \frac{\la^q}{\|f_{m-1}\|^q}\right) + \frac{ \la t(\|f_B\|+\e)}{\|f_{A(m)}\|_1}.
 $$ 
 Using definition of $A(m)$ we bound by {\bf A3}
 $$
 \|f_{A(m)}\|_1 = \sum_{i\in A(m)}|x_i| \le V|A(m)|^r\|f_{A(m)} + f-f_{A(m)} -G_{m-1}\|
 $$
 $$
 \le V|A|^r\|f-G_{m-1}\| \le V|A|^r(\|f_{m-1}\|+\e).
 $$
 This inequality is a variant of inequality (\ref{3.4a}) with $K$ replaced by $|A|$. Arguing as in the proof of Theorem \ref{T2.3} with $K$ replaced by $|A|$ we obtain the required inequality, which is the corresponding modification ($K$ is replaced by $|A|$ and $\e$ is replaced by $\|f_B\|+\e$) of (\ref{3.4b}). 
 
 The rest of the proof repeats the proof of Theorem \ref{T2.4} with the use of Lemma \ref{L5.2} instead of Lemma \ref{L5.1} and with the use of the fact that {\bf A3} implies {\bf A2} with $U=VK^r \le VK$.
 \end{proof}
 
 {\bf Proof of Theorem \ref{T2.6}.} This proof repeats the proof of Theorem \ref{T2.4} with the use of Lemma \ref{L5.2} instead of Lemma \ref{L5.1}.

 \section{Examples}

In this section, following \cite{Tsp}, we discuss applications of Theorems from Section 2 for specific dictionaries $\D$. Mostly, $\D$ will be a basis $\Psi$ for $X$. Because of that we use $m$ instead of $K$ in the notation of sparse approximation.  In some of our examples we take $X=L_p$, $2\le p<\infty$. Then it is known that $\rho(u) \le \gamma u^2$ with $\gamma = (p-1)/2$. In some other examples we take $X=L_p$, $1<p\le 2$. 
Then it is known that $\rho(u) \le \gamma u^{p}$,  with $\gamma = 1/p$.

\begin{Proposition}\label{Example 1.} Let $\Psi$ be a uniformly bounded orthogonal system normalized in $L_p(\Omega)$, $2\le p<\infty$, $\Omega$ is a bounded domain. Then we have 
 \begin{equation}\label{4.4}
\|f_{C(t,p,\Omega)m\ln (m+1)}\|_p \le C\sigma_m(f_0,\Psi)_p .
\end{equation}
\end{Proposition}
The proof of Proposition \ref{Example 1.} is based on   Theorem \ref{T2.7}.
\begin{Corollary}\label{Example 2.} Let $\Psi$ be the normalized in $L_p$, $2\le p<\infty$, real $d$-variate trigonometric
system. Then Proposition \ref{Example 1.} applies and gives  
for any $f_0\in L_p$
\begin{equation}\label{4.1}
\|f_{C(t,p,d)m\ln (m+1)}\|_p \le C\sigma_m(f_0,\Psi)_p .
\end{equation}
\end{Corollary}
We note that (\ref{4.1}) provides some progress in Open Problem 7.1 (p. 91) from \cite{T18}.

\begin{Proposition}\label{Example 1q.} Let $\Psi$ be a uniformly bounded orthogonal system normalized in $L_p(\Omega)$, $1< p\le 2$, $\Omega$ is a bounded domain. Then we have
 \begin{equation}\label{4.4q}
\|f_{C(t,p,\Omega)m^{p'-1}\ln (m+1)}\|_p \le C\sigma_m(f_0,\Psi)_p .
\end{equation}
\end{Proposition}
The proof of Proposition \ref{Example 1q.} is based on   Theorem \ref{T2.7}.

\begin{Corollary}\label{Example 2q.} Let $\Psi$ be the normalized in $L_p$, $1<p\le 2$, real $d$-variate trigonometric
system. Then Proposition \ref{Example 1q.} applies and gives  
for any $f_0\in L_p$
\begin{equation}\label{4.1q}
\|f_{C(t,p,d)m^{p'-1}\ln (m+1)}\|_p \le C\sigma_m(f_0,\Psi)_p .
\end{equation}
\end{Corollary}

\begin{Proposition}\label{Example 4.} Let $\Psi$ be the normalized in $L_p$, $2\le p<\infty$, multivariate Haar basis 
${\mathcal H}^d_p={\mathcal H}_p\times\cdots\times {\mathcal H}_p$. 
Then 
\begin{equation}\label{4.2}
\|f_{C(t,p,d)m^{2/p'}}\|_p \le C\sigma_m(f_0,{\mathcal H}^d_p)_p .
\end{equation}
\end{Proposition}
The proof of Proposition \ref{Example 4.} is based on Theorem \ref{T2.4}.
Inequality (\ref{4.2}) provides some progress in Open Problem 7.2 (p. 91) from \cite{T18} in the case $2<p<\infty$. 

 \begin{Proposition}\label{Example 4q.} Let $\Psi$ be the normalized in $L_p$, $1<p\le 2$, univariate Haar basis 
${\mathcal H}_p=\{H_{I,p}\}_I$, where $H_{I,p}$ the Haar functions indexed by dyadic intervals of support of $H_{I,p}$ (we index function $1$ by $[0,1]$ and the first Haar function by $[0,1)$). Then    
\begin{equation}\label{4.2q}
\|f_{C(t,p)m}\|_p \le C\sigma_m(f_0,{\mathcal H}_p)_p .
\end{equation}
\end{Proposition}
The proof of Proposition \ref{Example 4q.} is based on Theorem \ref{T2.8}.
Inequality (\ref{4.2q}) solves the Open Problem 7.2 (p. 91) from \cite{T18} in the case $1<p\le 2$.
 
 \begin{Proposition}\label{Example 5.} Let $X$ be a Banach space with $\rho(u)\le \gamma u^2$. Assume that $\Psi$ is a normalized Schauder basis for $X$. Then     
 \begin{equation}\label{4.5}
\|f_{C(t,X,\Psi)m^2\ln m}\| \le C\sigma_m(f_0,\Psi) .
\end{equation}
\end{Proposition}
The proof of Proposition \ref{Example 5.} is based on Theorem \ref{T2.7}.
We note that the above bound still works if we replace the assumption that $\Psi$ is a Schauder basis by the assumption that a dictionary $\D$ is $(1,D)$-unconditional with constant $U$. Then we obtain
$$
\|f_{C(t,\gamma,U)K^2\ln K}\| \le C\sigma_K(f_0,\Psi),\quad\text{for}\quad K+C(t,\gamma,U)K^2\ln K\le D .
$$

\begin{Proposition}\label{Example 5q.} Let $X$ be a Banach space with $\rho(u)\le \gamma u^q$, $1<q\le 2$. Assume that $\Psi$ is a normalized Schauder basis for $X$. Then     \begin{equation}\label{4.5q}
\|f_{C(t,X,\Psi)m^{q'}\ln m}\| \le C\sigma_m(f_0,\Psi) .
\end{equation}
\end{Proposition}
The proof of Proposition \ref{Example 5q.} is based on Theorem \ref{T2.7}.
We note that the above bound still works if we replace the assumption that $\Psi$ is a Schauder basis by the assumption that a dictionary $\D$ is $(1,D)$-unconditional with constant $U$. Then we obtain
$$
\|f_{C(t,\gamma,q,U)K^{q'}\ln K}\| \le C\sigma_K(f_0,\D),\quad\text{for}\quad K+C(t,\gamma,q,U)K^{q'}\ln K\le D .
$$

We now discuss application of general results of Section 2 to quasi-greedy bases. We begin with a brief introduction to the theory of quasi-greedy bases. Let $X$ be an
infinite-dimensional separable Banach space with a norm
$\|\cdot\|:=\|\cdot\|_X$ and let $\Psi:=\{\psi_m
\}_{m=1}^{\infty}$ be a normalized basis for $X$.      The concept of quasi-greedy basis was introduced in \cite{KonT}.

\begin{Definition}\label{D4.1}
The basis $\Psi$ is called quasi-greedy if there exists some
constant $C$ such that
$$\sup_m \|G_m(f,\Psi)\| \leq C\|f\|.$$
\end{Definition}

Subsequently, Wojtaszczyk \cite{W1} proved that these are
precisely the bases for which the TGA merely converges, i.e.,
$$\lim_{n\rightarrow \infty}G_n(f)=f.$$

 The following lemma is from \cite{DKK} (see also \cite{DS-BT} and \cite{GHO} for further discussions). 

\begin{Lemma}\label{L4.1} Let $\Psi$ be a quasi-greedy basis of $X$. Then for any finite set of indices $\Lambda$ we have for all $f\in X$
$$
\|S_\Lambda(f,\Psi)\| \le C \ln(|\Lambda|+1)\|f\|,  
$$
where for $f=\sum_{k=1}^\infty c_k(f)\psi_k$ we denote $S_\Lambda(f,\Psi) :=\sum_{k\in\Lambda} c_k(f)\psi_k$.
\end{Lemma}

We now formulate a result about quasi-greedy bases in $L_p$ spaces. The following theorem is from \cite{TYY1}. We note that in the case $p=2$
 Theorem \ref{T4.1} was proved in \cite{W1}. Some notations first. For a given element $f\in X$ we consider the expansion
$$
f=\sum_{k=1}^{\infty}c_k(f)\psi_k
$$
and the decreasing rearrangement of its coefficients
$$
|c_{k_1}(f)|\ge |c_{k_2}(f)|\ge... \,\,.
$$
Denote
$$a_n(f):=|c_{k_n}(f)|.$$

\begin{Theorem}\label{T4.1} Let $\Psi=\{\psi_m\}_{m=1}^\infty$ be a quasi-greedy basis of the $L_p$ space, $1<p<\infty$. Then for each $f\in X$ we have
$$
C_1(p)\sup_n n^{1/p}a_n(f)\le  \|f\|_p\le C_2(p) \sum_{n=1}^\infty n^{-1/2}a_n(f),\quad 2\le p <\infty;
$$
$$
C_3(p)\sup_n n^{1/2}a_n(f)\le  \|f\|_p\le C_4(p) \sum_{n=1}^\infty n^{1/p-1}a_n(f),\quad 1< p \le 2.
$$
\end{Theorem}

\begin{Proposition}\label{Example 6.}  
Let $\Psi$ be a normalized   quasi-greedy basis for $L_p$, $2\le p<\infty$.   Then  \begin{equation}\label{4.6}
\|f_{C(t,p)m^{2(1-1/p)} \ln (m+1)}\| \le C\sigma_m(f_0,\Psi) .
\end{equation}
\end{Proposition}
The proof of Proposition \ref{Example 6.} is based on Theorem \ref{T2.7}.

\begin{Proposition}\label{Example 6q.}  
Let $\Psi$ be a normalized   quasi-greedy basis for $L_p$, $1<p\le 2$.   Then \begin{equation}\label{4.6q}
\|f_{C(t,p)m^{p'/2} \ln (m+1)}\| \le C\sigma_m(f_0,\Psi) .
\end{equation}
\end{Proposition}
The proof of Proposition \ref{Example 6.} is based on Theorem \ref{T2.7}.

\begin{Proposition}\label{Example 7.} Let $\Psi$ be a normalized uniformly bounded orthogonal quasi-greedy basis for $L_p$, $2\le p<\infty$ (for existence of such bases see \cite{N}). Then  
\begin{equation}\label{4.7}
\|f_{C(t,p,\Psi)m  \ln\ln(m+3)}\|_p \le C\sigma_m(f_0,\Psi)_p .
\end{equation}
\end{Proposition}
The proof of Proposition \ref{Example 6.} is based on Theorem \ref{T2.8}.

\begin{Proposition}\label{Example 7q.} Let $\Psi$ be a normalized uniformly bounded orthogonal quasi-greedy basis for $L_p$, $1<p\le 2$ (for existence of such bases see \cite{N}). Then  
\begin{equation}\label{4.7q}
\|f_{C(t,p,\Psi)m^{p'/2}  \ln\ln(m+3)}\|_p \le C\sigma_m(f_0,\Psi)_p .
\end{equation}
\end{Proposition}
The proof of Proposition \ref{Example 6.} is based on Theorem \ref{T2.8}.

 Prposition \ref{Example 4q.} is the first result about almost greedy bases with respect to WCGA in Banach spaces.  It shows that the univariate Haar basis is an almost greedy basis with respect to the WCGA in the $L_p$ spaces for $1<p\le 2$. 
 Proposition \ref{Example 1.} shows that uniformly bounded orthogonal bases are $\phi$-greedy 
 bases with respect to WCGA with $\phi(u) = C(t,p,\Omega)\ln(u+1)$ in the $L_p$ spaces for $2\le p<\infty$. We do not know if these bases are almost greedy with respect to WCGA. They are good candidates for that. 

It is known (see \cite{Tbook}, p. 17) that the univariate Haar basis is a greedy basis with respect to TGA for all $L_p$, $1<p<\infty$. Proposition \ref{Example 4.} only shows that it is a $\phi$-greedy basis with respect to WCGA with $\phi(u) = C(t,p)u^{1-2/p}$ in the $L_p$ spaces for $2\le p<\infty$. It is much weaker than the corresponding results for the $\H_p$, $1<p\le 2$, and for the trigonometric system, $2\le p<\infty$ (see Corollary \ref{Example 2.}). We do not know if this result on the Haar basis can be substantially improved. At the level of our today's technique we can observe that the Haar basis is ideal (greedy basis) for the TGA in $L_p$, $1<p<\infty$, almost ideal (almost greedy basis) for the WCGA in $L_p$, $1<p\le 2$, and that the trigonometric system is very good for the WCGA in $L_p$, $2\le p<\infty$. 

Corollary \ref{Example 2q.}  shows that our results for the trigonometric system in $L_p$, $1<p<2$, are not as strong as for $2\le p<\infty$. We do not know if it is a lack of appropriate technique or it reflects the nature of the WCGA with respect to the trigonometric system. 

We note that Propositions \ref{P2.1} and \ref{P2.2} can be used to formulate the above Propositions for a more general bases. In these cases we use Propositions \ref{P2.1} and \ref{P2.2} with $D=\infty$. In Propositions \ref{Example 1.}, \ref{Example 1q.}, \ref{Example 6.}, and \ref{Example 6q.}, where we used Theorem \ref{T2.7}, we can replace the basis $\Psi$ by a basis $\Phi$, which dominates the basis $\Psi$. In Propositions \ref{Example 4.}, \ref{Example 4q.}, \ref{Example 7.}, and \ref{Example 7q.}, where we used either Theorem \ref{T2.4} or \ref{T2.8}, we can replace the basis $\Psi$ by a basis $\Phi$, which is equivalent to the basis $\Psi$.

It is interesting to compare Theorem \ref{T2.3} with the following known result. The following theorem provides rate of convergence (see \cite{Tbook}, p. 347). We denote by $A_1(\D)$ the closure in $X$ of the convex hull of the symmetrized dictionary $\D^\pm :=\{\pm g:g\in \D\}$. 
\begin{Theorem}\label{T5.1} Let $X$ be a uniformly smooth Banach space with modulus of smoothness $\rho(u)\le \gamma u^q$, $1<q\le 2$. Take a number $\e\ge 0$ and two elements $f_0$, $f^\e$ from $X$ such that
$$
\|f_0-f^\e\| \le \e,\quad
f^\e/A(\e) \in A_1(\D),
$$
with some number $A(\e)>0$.
Then, for the  WCGA    we have  
$$
\|f^{c,t}_m\| \le  \max\left(2\e, C(q,\gamma)(A(\e)+\e)t(1+m)^{1/q-1}\right) . 
$$
\end{Theorem}
Both   Theorem \ref{T5.1} and Theorem \ref{T2.3} provide stability of the WCGA with respect to noise. 
  In order to apply them for noisy data we interpret $f_0$ as a noisy version of a signal and $f^\e$ as a noiseless version of a signal. Then, assumption $f^\e/A(\e)\in A_1(\D)$ describes our smoothness assumption on the noiseless signal and assumption $f^\e \in \Sigma_K(\D)$ describes our structural assumption on the noiseless signal. 
  In fact, Theorem \ref{T5.1}  simultaneously takes care of two issues: noisy data and approximation in an interpolation space.
  Theorem \ref{T5.1} can be applied for approximation of $f_0$ under assumption that $f_0$ belongs  to one  
of interpolation spaces between $X$ and the space generated by the $A_1(\D)$-norm (atomic norm).

\section{Sparse trigonometric approximation}

Sparse trigonometric approximation of periodic functions began by the paper of S.B. Stechkin \cite{Ste}, who used it in the criterion for absolute convergence of trigonometric series. R.S. Ismagilov \cite{I} found nontrivial estimates for $m$-term approximation of functions with singularities of the type $|x|$ and gave interesting and important applications to the widths of Sobolev classes. He used a deterministic method based on number theoretical constructions. 
His method was developed by V.E. Maiorov \cite{M}, who used a method based on Gaussian sums.  Further strong results were obtained in \cite{DT1} with the help of a nonconstructive result from finite dimensional Banach spaces due to E.D. Gluskin \cite{G}. Other powerful nonconstructive method, which is based on a probabilistic argument, was used by Y. Makovoz \cite{Mk} and by E.S. Belinskii \cite{Be1}. Different methods were created in \cite{T29}, \cite{KTE1}, \cite{T1a}, \cite{Tappr} for proving lower bounds for function classes.
It was discovered in \cite{DKTe} and \cite{T12} that greedy algorithms can be used for constructive $m$-term approximation with respect to the trigonometric system. We demonstrate in \cite{T150} how greedy algorithms can be used to prove optimal or best known upper bounds for $m$-term approximation of classes of functions with mixed smoothness.  
It is a simple and powerful method of proving upper bounds.   The reader can find a detailed study of $m$-term approximation of classes of functions with mixed smoothness, including small smoothness, in 
the paper \cite{Rom1} by A.S. Romanyuk and in \cite{T150}, \cite{T152}. We note that in the case $2<p<\infty$ the upper bounds in \cite{Rom1} are not constructive.

We discuss some approximation problems for classes of functions with mixed smoothness. We define these classes momentarily.  
We will begin with the case of univariate periodic functions. Let for $r>0$ 
\be\label{6.3}
F_r(x):= 1+2\sum_{k=1}^\infty k^{-r}\cos (kx-r\pi/2) 
\ee
and
\be\label{6.4}
W^r_p := \{f:f=\varphi \ast F_r,\quad \|\varphi\|_p \le 1\}.  
\ee
It is well known that for $r>1/p$ the class $W^r_p$ is embedded into the space of continuous functions $C(\T)$. In a particular case of $W^1_1$ we also have embedding into $C(\T)$.

In the multivariate case for $\bx=(x_1,\dots,x_d)$ denote
$$
F_r(\bx) := \prod_{j=1}^d F_r(x_j)
$$
and
$$
\bW^r_p := \{f:f=\varphi\ast F_r,\quad \|\varphi\|_p \le 1\}.
$$
For $f\in \bW^r_p$ we will denote $f^{(r)} :=\varphi$ where $\varphi$ is such that $f=\varphi\ast F_r$.

The main results of Section 2 of \cite{T150} are the following two theorems. We use the notation 
$\beta:=\beta(q,p):= 1/q-1/p$ and $\eta:=\eta(q):= 1/q-1/2$. In the case of trigonometric system $\Tr^d$ we drop it from the notation:
$$
\sigma_m(\bW)_p := \sigma_m(\bW,\Tr^d)_p.
$$
\begin{Theorem}\label{T2.8I} We have
$$
 \sigma_m(\bW^r_q)_{p}
  \asymp  \left\{\begin{array}{ll} m^{-r+\beta}(\log m)^{(d-1)(r-2\beta)}, & 1<q\le p\le 2,\quad r>2\beta,\\
 m^{-r+\eta}(\log m)^{(d-1)(r-2\eta)}, & 1<q\le 2\le p<\infty,\quad r>1/q,\\ 
 m^{-r}(\log m)^{r(d-1)}, & 2\le q\le p<\infty, \quad r>1/2.\end{array} \right.
$$
\end{Theorem}
\begin{Theorem}\label{T2.9I} We have
$$
 \sigma_m(\bW^r_q)_{\infty}
  \ll  \left\{\begin{array}{ll}   m^{-r+\eta}(\log m)^{(d-1)(r-2\eta)+1/2}, & 1<q\le 2,\quad r>1/q,\\ 
 m^{-r}(\log m)^{r(d-1)+1/2}, & 2\le q<\infty, \quad r>1/2.\end{array} \right.
$$
\end{Theorem}
The case $1<q\le p\le 2$ in Theorem \ref{T2.8I}, which corresponds to the first line, was proved in \cite{T29} (see also \cite{Tmon}, Ch.4). The proofs from \cite{T29} and \cite{Tmon} are constructive. In \cite{T150} we concentrate on the case $p\ge 2$. We use recently developed techniques on greedy approximation in Banach spaces to prove Theorems \ref{T2.8I} and \ref{T2.9I}. It is important that greedy approximation allows us not only to prove the above theorems but also to provide a constructive way for building the corresponding $m$-term approximants. We give a precise formulation from \cite{T150}.
\begin{Theorem}\label{T2.8C} For $p\in(1,\infty)$ and $\mu>0$ there exist constructive methods $A_m(f,p,\mu)$, which provide for $f\in \bW^r_q$ an $m$-term approximation such that
$$
 \|f-A_m(f,p,\mu)\|_p
 $$
 $$
  \ll  \left\{\begin{array}{ll} m^{-r+\beta}(\log m)^{(d-1)(r-2\beta)}, & 1<q\le p\le 2,\quad r>2\beta+\mu,\\
 m^{-r+\eta}(\log m)^{(d-1)(r-2\eta)}, & 1<q\le 2\le p<\infty,\quad r>1/q+\mu,\\ 
 m^{-r}(\log m)^{r(d-1)}, & 2\le q\le p<\infty, \quad r>1/2+\mu.\end{array} \right.
$$
\end{Theorem}
Similar modification of Theorem \ref{T2.9I} holds for $p=\infty$.
We do not have matching lower bounds for the upper bounds in Theorem \ref{T2.9I} in the case of approximation in the uniform norm $L_\infty$. 
 
As a direct corollary of Theorems \ref{T1.2} and \ref{T2.8I} we obtain the following result.
\begin{Theorem}\label{TCW} Let $p\in [2,\infty)$. Apply the WCGA with weakness parameter $t\in (0,1]$ to $f\in L_p$ with respect to the real trigonometric system ${\mathcal R}{\mathcal T}_p^d$. If $f\in \bW^r_q$, then we have
$$
 \|f_m\|_{p}
  \ll  \left\{\begin{array}{ll}  
 m^{-r+\eta}(\log m)^{(d-1)(r-2\eta)+r-\eta}, & 1<q\le 2,\quad r>1/q,\\ 
 m^{-r}(\log m)^{rd}, & 2\le q<\infty, \quad r>1/2.\end{array} \right.
$$
\end{Theorem}

The reader can find results on best $m$-term approximation as well as results on constructive $m$-term approximation for Besov type classes in the paper \cite{T150}. Some new results in the case of small smoothness are contained in \cite{T152}. 

\section{Tensor product approximations}

In the paper \cite{BT} we study multilinear approximation (nonlinear tensor product approximation) of functions.
 For a function $f(x_1,\dots,x_d)$ denote
$$
\Th_M(f)_X:=\inf_{\{u^i_j\},  j=1,\dots,M, i=1,\dots,d}\|f(x_1,\dots,x_d) - \sum_{j=1}^M\prod_{i=1}^d u^i_j(x_i)\|_X
$$
and for a function class $F$ define
$$
\Th_M(F)_X:=\sup_{f\in F} \Th_M(f)_X.
$$
In this section we use the notation $M$ instead of $m$ for the number of terms in an approximant because this notation is a standard one in the area. In the case $X=L_p$ we write $p$ instead of $L_p$ in the notation.
In other words we are interested in studying $M$-term approximations of functions with respect to the dictionary
$$
\Pi^d := \{g(x_1,\dots,x_d): g(x_1,\dots,x_d)=\prod_{i=1}^d u^i(x_i)\}
$$
where $u^i(x_i)$ are arbitrary univariate functions. 
We discuss the case of $2\pi$-periodic functions of $d$ variables and approximate them in the $L_p$ spaces. Denote by $\Pi^d_p$ the normalized in $L_p$ dictionary $\Pi^d$ of $2\pi$-periodic functions. We say that a dictionary $\D$ has a tensor product structure if all its elements have a form of products $u^1(x_1)\cdots u^d(x_d)$ of univariate functions $u^i(x_i)$, $i=1,\dots,d$. Then any dictionary with tensor product structure is a subset of $\Pi^d$. 
The classical example of a dictionary with tensor product structure is the $d$-variate trigonometric system $\{e^{i(\bk,\bx)}\}$. Other examples include the hyperbolic wavelets and the hyperbolic wavelet type system $\U^d$ defined in \cite{T69}. 

Modern problems in approximation, driven by applications in biology, medicine, and engineering, are being formulated in very high dimensions, which brings to the fore new phenomena. For instance, partial differential equations in a phase space of large spacial dimensions (e.g. Schr\"odinger and Fokker-Plank equations) are very important in applications. It is known (see, for instance, \cite{DDGS}) that such equations involving large number of spacial variables pose a serious computational challenge because of the so-called {\it curse of dimensionality}, which is caused by the use of classical notions of smoothness as the regularity characteristics of the solution. The authors of \cite{DDGS} show that replacing the classical smoothness assumptions by structural assumptions in terms of sparsity with respect to the dictionary $\Pi^d$, they overcome the above computational challenge. They prove that the solutions of certain high-dimensional equations inherit sparsity, based on tensor product decompositions, from given data. Thus, our algorithms, which provide good sparse approximation with respect to $\Pi^d$ for individual functions might be useful in applications to PDEs of the above type. 
The nonlinear tensor product approximation is very important in numerical applications. We refer the reader to the monograph \cite{H} which presents the state of the art on the topic. Also, the reader can find a very recent discussion of related results in \cite{SU}. 

In the case $d=2$ the multilinear approximation problem is the classical problem of bilinear approximation. In the case of approximation in the $L_2$ space the bilinear approximation problem is closely related to the problem of singular value decomposition (also called Schmidt expansion) of the corresponding integral operator with the kernel $f(x_1,x_2)$. There are known results on the rate of decay of errors of best bilinear approximation in $L_p$ under different smoothness assumptions on $f$. We only mention some known results for classes of functions with mixed smoothness. We study the classes $\bW^r_q$ of functions with bounded mixed derivative   defined above in Section 5.

The problem of estimating $\Th_M(f)_2$ in case $d=2$ (best $M$-term bilinear approximation in $L_2$) is a classical one and was considered for the first time by E. Schmidt \cite{S} in 1907. For many function classes $F$ an asymptotic behavior of
$\Th_M(F)_p$ is known. For instance, the relation
\begin{equation}\label{1.1}
\Th_M(\bW^r_q)_p \asymp M^{-2r + (1/q-\max(1/2,1/p))_+}
\end{equation}
for $r>1$ and $1\le q\le p \le \infty$ follows from more general results in \cite{T32}.
In the case $d>2$ almost nothing is known. There is (see \cite{T35}) an upper estimate in the case $q=p=2$
\begin{equation}\label{1.2}
\Th_M(\bW^r_2)_2 \ll M^{-rd/(d-1)} .  
\end{equation}

Results of \cite{BT} are around the bound (\ref{1.2}). First of all we discuss the lower bound 
matching the upper bound (\ref{1.2}). In the case $d=2$ the lower bound 
\begin{equation}\label{1.3}
\Th_M(W^r_p)_p \gg M^{-2r},\qquad 1\le p\le \infty ,  
\end{equation}
follows from more general results in \cite{T32} (see (\ref{1.1}) above). A stronger result 
\begin{equation}\label{1.4}
\Th_M(W^r_\infty)_1 \gg M^{-2r}  
\end{equation}
follows from Theorem 1.1 in \cite{T46}. 

We could not prove the lower bound matching the upper bound (\ref{1.2}) for $d>2$. Instead, we proved in \cite{BT} a weaker lower bound. For a function $f(x_1,\dots,x_d)$ denote
$$
\Th^b_M(f)_X:=\inf_{\{u^i_j\}, \|u^i_j\|_X\le b\|f\|_X^{1/d}}\|f(x_1,\dots,x_d) - \sum_{j=1}^M\prod_{i=1}^d u^i_j(x_i)\|_X
$$
and for a function class $F$ define
$$
\Th^b_M(F)_X:=\sup_{f\in F} \Th^b_M(f)_X.
$$
In \cite{BT} we proved the following lower bound (see Corollary \ref{C2.2})
$$
\Th^b_M(\bW^r_\infty)_{ 1} \gg (M\ln M)^{-\frac{rd}{d-1}}.
$$
This lower bound indicates that probably the exponent $\frac{rd}{d-1}$ is the right one in the power decay of the $\Th_M(\bW^r_p)_p$. 

Secondly, we discuss some upper bounds which extend the bound (\ref{1.2}). The relation (\ref{1.1}) shows that for $2\le p\le \infty$ in the case $d=2$ one has
\begin{equation}\label{1.5}
\Th_M(\bW^r_2)_p \ll M^{-2r} .  
\end{equation}

In \cite{BT} we extend (\ref{1.5}) for $d>2$.
\begin{Theorem}\label{T1.1} Let $2\le p<\infty$ and $r> (d-1)/d$. Then
$$
\Th_M(\bW^r_2)_p \ll  \left(\frac{M}{(\log M)^{d-1}}\right)^{-\frac{rd}{d-1}}.
$$
\end{Theorem}

The proof of Theorem \ref{T1.1} in \cite{BT} is not constructive. It goes by induction and uses 
a nonconstructive bound in the case $d=2$, which is obtained in \cite{T35}. The corresponding proof from \cite{T35} uses the bounds for the Kolmogorov width $d_n(W^r_2, L_\infty)$, proved by Kashin \cite{Ka}. Kashin's proof is a probabilistic one, which provides existence of a good linear subspace for approximation, but there is no known explicit constructions of such subspaces. This problem is related to a problem from compressed sensing  on construction of good matrices with Restricted Isometry Property. It is an outstanding difficult open problem.  
In \cite{BT} we discuss constructive ways of building good multilinear approximations. The simplest way would be to use known results about $M$-term approximation with respect to special systems with tensor product structure. However, this approach (see \cite{T69}) provides error bounds, which are not as good as best $m$-term approximation with respect to $\Pi^d$ (we have exponent $r$ instead of $\frac{rd}{d-1}$ for $\Pi^d$). It would be very interesting  to provide a constructive multilinear approximation method with the same order of the error as the best $m$-term approximation.  

As we pointed out in a discussion of Theorem \ref{T1.1} the upper bound in Theorem \ref{T1.1} is proved with a help of probabilistic results. There is no known deterministic constructive methods (theoretical algorithms), which provide the corresponding upper bounds.
In  \cite{BT} we apply greedy type algorithms to obtain upper estimates of $\Theta_M(\bW^r_2)_p$. The important feature of our proof is that it is deterministic and moreover it is constructive. Formally, the optimization problem
$$
\Th_M(f)_X:=\inf_{\{u^i_j\},  j=1,\dots,M, i=1,\dots,d}\|f(x_1,\dots,x_d) - \sum_{j=1}^M\prod_{i=1}^d u^i_j(x_i)\|_X
$$
is deterministic: one needs to minimize over $u^i_j$. However,  minimization by itself does not provide any upper estimate. It is known (see \cite{DMA}) that simultaneous optimization over many parameters is a very difficult problem. Thus, in nonlinear $M$-term approximation we look for methods (algorithms), which provide approximation close to best $M$-term approximation and at each step solve an optimization problem over only one parameter ($\prod_{i=1}^d u^i_j(x_i)$ in our case). In \cite{BT} we  provide such an algorithm for estimating $\Theta_M(f)_p$. We call this algorithm {\it constructive} because it provides an explicit construction with feasible one parameter optimization steps. We stress that in the setting of approximation in an infinite dimensional Banach space, which is considered in \cite{BT}, the use of term {\it algorithm} requires some explanation. In that paper we discuss only theoretical aspects of the efficiency (accuracy) of $M$-term approximation and possible ways  to  realize this efficiency.    The {\it greedy algorithms} used in \cite{BT} give a procedure to construct an approximant, which turns out to be a good approximant. The procedure of
constructing a greedy approximant is not a numerical algorithm ready for computational implementation. Therefore, it would be more precise to call this procedure a {\it theoretical greedy algorithm} or {\it stepwise optimizing process}. Keeping this remark in mind we, however, use the term {\it greedy algorithm} in this paper because it has been used in previous papers and has become a standard name for procedures used in \cite{BT} and for more general procedures of this type (see for instance
\cite{D}, \cite{Tbook}). Also, the theoretical algorithms, which we use in \cite{BT}, become 
algorithms in a strict sense if instead of an infinite dimensional setting we consider a finite dimensional setting, replacing, for instance, the $L_p$ space by its restriction on the set of trigonometric polynomials. We note that the greedy-type algorithms are known to be very efficient in numerical applications (see, for instance, \cite{Za} and \cite{SSZ}).

In \cite{BT} we use two very different greedy-type algorithms to provide a constructive multilinear approximant. The first greedy-type algorithm  is based on a very simple dictionary   consisting of shifts of the de la Vall{\'e}e Poussin kernels. The algorithm uses function (dyadic blocks of a function) evaluations  and picks the largest of them. The second greedy-type algorithm  is more complex. It is based on 
the dictionary $\Pi^d$ and uses the Weak Chebyshev Greedy Algorithm with respect to $\Pi^d$ to update the approximant. 
Surprisingly, these two algorithms give the same error bound. For instance, Theorems 4.3 and 4.4 from \cite{BT} give for big enough $r$ the following constructive upper bound for $2\le p<\infty$
$$
\Th_M(\bW^r_2)_p \ll \left(\frac{M}{(\ln M)^{d-1}}\right)^{-\frac{rd}{d-1} +\frac{\bt}{d-1}},\quad \bt:=\frac{1}{2}-\frac{1}{p} .
$$
This constructive upper bound has an extra term $\frac{\bt}{d-1}$ in the exponent compared to the best $M$-term approximation. It would be interesting to find a constructive proof of Theorem \ref{T1.1}.

\end{document}